\documentclass{amsart}

\newif\ifdraft\draftfalse
\newif\ifcite\citefalse
\newif\ifblow\blowfalse

\makeatletter
\@namedef{subjclassname@2020}{\textup{2020} Mathematics Subject Classification}
\makeatother

\usepackage{amsfonts,amssymb, amsmath}
\usepackage{url}
\usepackage{amscd}
\usepackage{mathrsfs} 
\usepackage[hyperindex,hidelinks]{hyperref}  
\newcommand{\cites}[1]{\cite{#1}}
\ifdraft\ifcite\usepackage{showkeys}\else\usepackage[notcite,notref]{showkeys}\fi\fi

\newtheorem{proposition}[equation]{Proposition}
\newtheorem{theorem}[equation]{Theorem}

\newtheorem{lemma}[equation]{Lemma}

\theoremstyle{definition}
\newtheorem{definition}[equation]{Definition}

\theoremstyle{remark}

\newtheorem{remark}[equation]{Remark}

\newtheorem{example}[equation]{Example}

\numberwithin{equation}{section}

\newcommand\beq{\begin{equation}}
\newcommand\eeq{\end{equation}}

\newcommand\p{\partial}

\newcommand{\R}{\mathbb R}
\newcommand{\C}{\mathbb C}

\newcommand{\Id}{\mathrm{Id}}
\newcommand{\eps}{\varepsilon}

\newcommand{\diag}{\operatorname{diag}}
\newcommand{\Inv}{\operatorname{Inv}}

\newcommand{\Wt}{\operatorname{Wt}}

\newcommand\al{\alpha}

\newcommand\ld{\lambda}
\newcommand\om{\omega}

\newcommand\lb{\label}

\newcommand\er{\eqref}

\newcommand\ol{\overline}

\newcommand\on{\operatorname}
\newcommand\la{\langle}
\newcommand\ra{\rangle}

\newcommand\lp{\left(}
\newcommand\rp{\right)}

\newcommand\cc{{\mathcal C}}

\newcommand\cf{{\mathcal F}}

\newcommand\ck{{\mathcal K}}

\newcommand\bc{{\mathbb C}}

\newcommand\br{{\mathbb R}}

\newcommand\SL{\mathrm{SL}}

\newcommand\fg{\mathfrak{g}}
\newcommand\fh{\mathfrak{h}}

\newcommand\fl{\mathfrak{l}}

\newcommand\fn{\mathfrak{n}}

\newcommand\fp{\mathfrak{p}}

\newcommand\fz{\mathfrak{z}}

\newcommand\wtjj{\Wt(V^J_{\omega_j})}

\title
[Weyl chambers and Toda bubbling]
{Weyl Chambers, Asymptotic Expansions, and Blowup Analysis of Explicit Toda Solutions}
\author{Zhaohu Nie}
\address{Department of Mathematics and Statistics, Utah State University, Logan, Utah 84322-3900, USA}
\email{zhaohu.nie@usu.edu}

\begin{document}

\begin{abstract}
Blowup analysis is a central part of geometric analysis, and systems of
PDEs such as Toda systems exhibit particularly rich blowup behavior.
In \cite{IMRN}, the author constructed explicit families of radial Toda
solutions whose blowup masses correspond to arbitrary elements of the
Weyl group of the underlying Lie algebra. In this paper, we study these
families within the framework of blowup analysis developed by C.-S. Lin
and his coauthors. The key is an
asymptotic analysis of the solution family that distinguishes Cartan elements in the interior
of a Weyl chamber from those lying on its boundary faces.
We prove a uniform upper bound for $u_i^{\lambda,H}(r)+2\log r$,
so that the selection estimate holds with the center at the origin. We show that the successive bubbling scales
for slow decay are determined by the
intersections of the ray $H+tq_\mu$ with root hyperplanes. 
At each
crossing, rescaling and the face asymptotics identify the limiting
profiles as explicit solutions of the corresponding Levi--Toda
subsystems. 
The masses of these successive bubbles, computed using
\cite{KLNW}, account for the entire blowup mass.
Examples in type $A_2$ demonstrate this identification and exhibit both separated scalar bubbles and a
simultaneous full-system bubble.
\end{abstract}

\maketitle

\section{Introduction}

For a complex semisimple Lie algebra $\fg$ of rank $n$ with Cartan
matrix $A=(a_{ij})$, consider the singular Toda system
\begin{equation}\label{Toda}
\begin{cases}
\displaystyle
\Delta u_i+4\sum_j a_{ij}e^{u_j}
   =4\pi\gamma_i\delta_0
   &\text{in }B_2(0),\\[4pt]
\displaystyle
\int_{B_2(0)}e^{u_i}<\infty,
   \qquad \gamma_i>-1,
   &1\le i\le n,
\end{cases}
\end{equation}
where $\delta_0$ is the Dirac delta measure at the origin of the plane, and $B_2(0)$ is the disk centered at the origin with radius 2. In this paper, $\sum_j$ denotes $\sum_{j=1}^n$, unless otherwise indicated.

We work with solutions smooth on $B^*_2:=B_2(0)\setminus\{0\}$ whose regular parts
\beq\lb{reg part}
\hat u_i(z)=u_i(z)-2\gamma_i\log|z|
\eeq
extend continuously across the origin. 
A sequence $u^k=(u_1^k,\ldots,u_n^k)$ of solutions of
\er{Toda} is said to blow up at the origin if there exist
points $x_k\to 0$ such that
\begin{equation}\label{blowup}
\lim_{k\to\infty}\max_{1\le i\le n}\hat u_i^k(x_k)=+\infty.
\end{equation}
The origin is the only blowup point of the sequence if furthermore  
$$
\max_{1\leq i\leq n}\sup_{x\in \cc} \hat u_i^k(x)\leq C_{\cc},\quad \forall\,\cc\Subset B^*_2\text{ compact subset}.
$$
We are concerned with this case of blowup only at the origin.

For a solution $u=(u_1,\ldots,u_n)$ of \eqref{Toda}, define
\beq\lb{enclosed}
\sigma_i(u;r)=\frac{1}{\pi}\int_{B_r(0)}e^{u_i},
\qquad 0<r<2.
\eeq
Note that the requirement that $\gamma_i>-1$ guarantees that the $e^{u_i}$ are locally integrable
at the origin. 
For a blowup sequence $u^k$ with uniformly bounded masses,
upon passing to a subsequence, we define its local mass vector 
$\sigma=(\sigma_1,\ldots,\sigma_n)$ at the origin by
\begin{equation}\lb{local mass}
\sigma_i(u)
=\lim_{r\downarrow0}\lim_{k\to\infty}\sigma_i(u^k;r)
=\frac{1}{\pi}\lim_{r\downarrow0}\lim_{k\to\infty}
  \int_{B_r(0)}e^{u_i^k}\,dx.
\end{equation}

The study of Toda systems, which are integrable systems, is closely tied to the underlying Lie algebra. 
We now introduce the Lie-theoretic setup and refer the reader to \cite{Knapp} for more details. 
Let $\fh$ be a Cartan subalgebra of the semisimple $\fg$, and $\Phi$ the root system of
$\fg$ with respect to $\fh$. 
Let $\fh_0'$ be the real span of the roots, and let $\fh_0$ be the
real form of $\fh$ on which every root is real-valued. 
We denote by $\la\cdot,\cdot\ra: \fh_0'\times\fh_0\to \br$ the canonical pairing, and by $(\cdot,\cdot)$ the Killing form
on $\fh_0$ and its dual form on $\fh_0'$.

Fix a decomposition $\Phi =\Phi^+\sqcup\Phi^-$ of the roots into positive and negative ones. 
Let
$\Delta=\{\alpha_1,\ldots,\alpha_n\}$ be the simple roots, and 
$\omega_1,\ldots,\omega_n$ the fundamental weights determined by 
$
2\frac{(\omega_i, \alpha_j)}{(\al_j, \al_j)} = \delta_{ij}.
$
With the convention $a_{ij} = 2\frac{(\alpha_i, \alpha_j)}{(\al_j, \al_j)}$, we thus have
\beq\lb{raw}
\alpha_i=\sum_{j=1}^n a_{ij}\omega_j, \quad i=1,\dots, n.
\eeq

Let $W$ be the corresponding Weyl group \cite{Knapp}, acting on both $\fh_0$ and $\fh_0'$. 
Let $s_i$ denote the reflection in the simple root $\alpha_i$. 
These reflections generate $W$. 
The smallest number of such simple reflections required in an element $w\in W$ is called its length $\ell(w)$. 
The unique element of $W$ with the maximum length is denoted by $\kappa$, which is characterized by the fact that it maps 
all the positive roots to negative roots. 

Let
\[
C_0=\{X\in\fh_0:\langle\alpha_i,X\rangle>0,\ 1\le i\le n\}
\]
be the dominant Weyl chamber. Set $\mu_i=1+\gamma_i>0$, and let
$q_\mu\in C_0$ be the element uniquely determined by
\beq\lb{def q}
\langle\alpha_i,q_\mu\rangle=\mu_i,
\qquad 1\le i\le n.
\eeq
Therefore the element $q_\mu\in \fh_0$ determines the singularity strengths  at the origin. Later on, this element will go through subtle changes during our bubbling procedure (see \er{strength}), which determines the singularity strength for the subsystems along the way.  
(This element was denoted by $w_0$ in \cites{KLNW, IMRN} following \cite{Kostant};
we use $q_\mu$ here to distinguish it from Weyl-group elements $w\in W$.)

For $w\in W$, define its Weyl-group mass vector by
\beq\lb{mass for w}
\sigma_i(w)=\langle\omega_i-w\omega_i,q_\mu\rangle,
\eeq
where $\omega_i$ is the $i$th fundamental weight. 

In \cite{IMRN}, the author constructed, for every $H\in \fh_0$, 
a family of singular Toda solutions $u^{\lambda,H}$ with $\ld>0$. These solutions are radial, and we write $u_i^{\lambda,H}(r)$ with
$r=|z|$. Section~\ref{sec:weights} recalls their exact formulas in \er{eq:Ui-ui}. 

In \cite{IMRN}, the author proved that for every $\tau\in W$ and $H\in \tau C_0$, the local masses satisfy
\begin{equation}\label{eq:endpoint-intro}
\sigma_i
=\lim_{r\downarrow0}\lim_{\lambda\to\infty}
  \sigma_i(u^{\lambda,H};r)
=\sigma_i(\tau) =\langle\omega_i-\tau\omega_i,q_\mu\rangle.
\end{equation}

We also mention two related works of the author \cite{affine,coco}.
In \cite{affine}, we give the affine analogue of the constructions in \cite{IMRN}.
In \cite{coco}, we prove that, under the hypotheses stated there,
all possible blowup masses for the constant-coefficient Toda system
\er{Toda} correspond to Weyl group elements.
Thus, the examples in \cite{IMRN} realize all possible blowup masses
in that setting.

We now introduce the blowup-analysis framework developed in
\cites{LS, JLW,LWZ,LWYZ,LYZ, CNY}.  
The central questions in blowup analysis are to identify the concentration points, 
find the bubbling
scales, determine the limiting system at each scale, and describe
how the local masses change from one scale to the next.

The first step in this framework is to select a finite set
$\Sigma_\lambda$ such that
\[
u_i^{\lambda, H}(z)
+2\log\on{dist}(z,\Sigma_\lambda)\leq C,
\qquad 1\leq i\leq n,
\]
where $C$ is independent of $\lambda$ and $z$. For our explicit family, we prove
this estimate with $\Sigma_\lambda=\{0\}$ for every $\lambda$.
The following quantity is of central importance to this paper: 
\beq\lb{def f}
 F_i^{\lambda,H}(r)=u_i^{\lambda,H}(r)+2\log r.
 \eeq

\begin{theorem}[Uniform selection bound]\label{just 0}
Let $\kappa$ be the longest element of $W$ and set
$\mathcal M_i=\sigma_i(\kappa)=\langle\omega_i-\kappa\omega_i,q_\mu\rangle$.
For every $H\in\fh_0$, $\lambda>0$, and $r>0$,
$$F_i^{\lambda,H}(r)\leq2\log\frac{\mathcal M_i} 2,$$
so the selection estimate holds with $\Sigma_\lambda=\{0\}$.
\end{theorem}

Next, we study the bubbling scales using the following concept of fast and slow 
decay introduced in the works \cite{LWZ, LWYZ, LYZ}. 
We study families of radii $r_\lambda$ depending on $\lambda$, particularly the
exponential scales $r_\lambda^t=e^{-\lambda t}$ indexed by $t>0$.

\begin{definition}\label{decays}
For a radius family $r_\lambda$, component $i$ has \emph{fast decay}
at $r_\lambda$ if
\[
F_i^{\lambda,H}(r_\lambda)\to-\infty
\quad\text{as }\lambda\to\infty.
\]
Otherwise, it has \emph{slow decay} at $r_\lambda$.

The solution family $u^{\lambda, H}$ has \emph{fast decay} at $r_\lambda$
if every component has fast decay there.
Otherwise, it has \emph{slow decay} at $r_\lambda$.
\end{definition}

The importance of fast and slow decay is as follows. The Toda system admits the following scaling. If $u_i$ is a solution of \er{Toda} on $\bc^*$, then for $\eps>0$, 
\beq\lb{scaling}
\widetilde u_i(y) = u_i(\eps y) + 2\log \eps
\eeq
is also a solution. 
For the explicit family considered here, we will show that at the
scales identified below, the fast components tend to $-\infty$ after
rescaling, while the slow components converge to solutions of a subsystem. 

\begin{example}[The scalar model] To make the reader more familiar with the construction of this paper, we first consider the simplest Liouville equation, the example of Toda system \er{Toda} for the Lie algebra $A_1$. 
For the nonsingular Liouville equation $\Delta u+8e^u=0$ with $\gamma=0$, consider the family 
\[
 u^k(r)=2\log\frac{k}{1+k^2r^2},\quad k>0.
\]
($k=e^\lambda$ in our general notation.) 
Then $u^k(0)=2\log k\to\infty$ and
\[
 u^k(r)+2\log r=2\log\frac{kr}{1+k^2r^2}\leq 2\log\frac 1 2,
\]
confirming Theorem~\ref{just 0}. 

The radius $r=k^{-1}$ is slow, and its rescaling \er{scaling} is
\begin{equation*}
\widetilde u(\rho)= u^k(k^{-1}\rho)-2\log k \longrightarrow -2\log(1+\rho^2),
\end{equation*}
the standard solution to the Liouville equation. 
For each fixed $a>0$, $a\ne1$, the radius $r=k^{-a}$ is fast.

The enclosed \er{enclosed} and local masses \er{local mass} are
\[
\sigma(u^k; r)= \frac1\pi\int_{B_r}e^{u^k}
 =\frac{k^2r^2}{1+k^2r^2},\qquad
 \sigma=\lim_{r\to 0}\lim_{k\to \infty} \sigma(u^k, r) = 1.
\]
This is also the mass of the standard solution on the whole plane
$$
\frac{1}{\pi}\int_{\br^2}  e^{\widetilde u} = 1.
$$

This example illustrates the general procedure of locating slow decay radii, finding the limiting bubbles, and computing their mass increments, which we will carry out for our explicit families $u^{\ld, H}$ for Toda systems of all Lie types. 
\end{example}

The geometry of the fast and slow scales for our solution family  is described by Weyl
chambers and their faces.
Define the relatively open face
\beq\lb{fj}
 \mathcal F_J=\{X\in\overline{C_0}:
 \langle\alpha_j,X\rangle=0, j\in J;\quad
 \langle\alpha_i,X\rangle>0, i\notin J\}.
\eeq
For $J=\emptyset$, we have $\cf_\emptyset = C_0$, the interior. 
These faces give a stratification
\[
\overline{C_0}=\bigsqcup_{J\subset\{1,\ldots,n\}}\mathcal F_J.
\]

Relatedly, for $J\subset\{1,\ldots,n\}$, we use $\fl_J$ to denote the reductive Levi subalgebra of the parabolic subalgebra $\fp_J$ defined by $J$ (see, for example, \cite{Knapp}). Then 
\beq\lb{lg}
\fl_J = \fg_J\oplus \fz_{J},
\eeq
where 
\beq\lb{cx ct}
\fz_{J}=\{Z\in\fh:\langle\alpha_i,Z\rangle=0,\ i\in J\}
\eeq
is the center of $\fl_J$, and the semisimple part $\fg_J=[\fl_J,\fl_J]$ is generated by the simple-root generators
indexed by $J$. 
The root system of $\fg_J$ is
$\Phi_J=\Phi\cap\operatorname{span}_{\mathbb Z}\{\alpha_i:i\in J\}$,
and its Cartan matrix is $A_J=(a_{ij})_{i,j\in J}$. Its Dynkin diagram is the full subdiagram of that of $\fg$ on
$J$, which may be disconnected. Its Weyl group is
$W_J=\langle s_i:i\in J\rangle$, with longest element $\kappa_J$. We call the Toda system for $\fg_J$ the \emph{Levi-Toda system} associated to the subset $J$. 

Our main result is the following theorem that identifies every bubbling scale and its limiting profile. 

\begin{theorem}[Bubbling decomposition]\label{thm:main}
Let 
\beq\lb{def H}
\tau \in W,\ H_0\in C_0,  \text{ and }H=\tau H_0.
\eeq 
Consider the explicit Toda family
$u^{\lambda,H}$ in \er{eq:Ui-ui}.  Set
\begin{equation*}
X_H(t)=H+tq_\mu,\qquad t\ge0.
\end{equation*}
\begin{enumerate}
\item
There are finitely many distinct crossing times
\[
0<t_1<\cdots<t_N
\]
(the list is empty if $N=0$)
at which $X_H(t)$ lies on a root hyperplane. 
Set $t_0=0$ and $t_{N+1}=\infty$, and define $w_a\in W$ by
\beq\lb{def w_a}
X_H(t)\in w_aC_0
\qquad\text{for }t_a<t<t_{a+1},
\qquad 0\le a\le N.
\eeq
Then $w_0=\tau$, $w_N=\Id$. 
For each $1\le a\le N$, there is a nonempty set
$J_a\subset\{1,\ldots,n\}$ such that
\begin{equation}\lb{wwk}
X_H(t_a)\in w_a\mathcal F_{J_a},
\qquad
w_{a-1}=w_a\kappa_{J_a}.
\end{equation}
Consequently,
\begin{equation*}
\tau=\kappa_{J_N}\cdots\kappa_{J_1}.
\end{equation*}
The factorization is length-additive:
\beq\lb{leng+}
 \ell(\tau)=\sum_{a=1}^N\ell(\kappa_{J_a}).
\eeq
If $H$ is generic in $\tau C_0$, then every $J_a$ is a singleton,
say $J_a=\{j_a\}$, and
\[
\tau=s_{j_N}\cdots s_{j_1}
\]
is a reduced expression.

\item 
For fixed $t\in(t_a,t_{a+1})$, put $r_\lambda^t=e^{-\lambda t}$.
Then every rescaled component satisfies
\[
 u_i^{\lambda,H}(r_\lambda^t y)+2\log r_\lambda^t
 \longrightarrow-\infty
 \quad\text{locally uniformly on }\C^*.
\]
In particular, all components have fast decay at $r=r_\lambda^t$. Furthermore, 
\begin{equation*}
\lim_{\lambda\to\infty}
\frac1\pi\int_{B_{e^{-\lambda t}}}
e^{u_i^{\lambda,H}}
=
\sigma_i(w_a).
\end{equation*}

\item At 
\begin{equation}\lb{aftall}
r=r_\lambda^{t_a}=e^{-\lambda t_a},
\qquad 1\le a\le N, 
\end{equation}
the slow components are exactly
those indexed by $J_a$.

Define
\beq\lb{a-scale}
\widetilde u_i^{\lambda,a}(y)
=
u_i^{\lambda,H}(r_\lambda^{t_a}y)
+2\log r_\lambda^{t_a}.
\eeq
Then 
$$
\widetilde u_i^{\lambda,a}
\longrightarrow
\begin{cases}
-\infty, & i\notin J_a,\\
v_i^{(a)}, & i\in J_a,
\end{cases}
$$
locally uniformly on $\C^*$, where 
$v^{(a)}=(v_i^{(a)})_{i\in J_a}$ is a radial finite-energy
solution of the singular
Levi--Toda system
\begin{equation}\label{eq:levi-main}
\Delta v_i^{(a)}
+4\sum_{j\in J_a}a_{ij}e^{v_j^{(a)}}
=
4\pi\gamma_i^{(a)}\delta_0
\quad\text{in }\R^2,
\qquad i\in J_a,
\end{equation}
where the strength of the singular source is given by 
\begin{equation}\lb{strength}
1+\gamma_i^{(a)}
=\langle\alpha_i, w_a^{-1}q_\mu\rangle
 =\langle w_a\alpha_i,q_\mu\rangle>0, \qquad i\in J_a.
\end{equation}
Furthermore, the $v^{(a)}$ are the explicit Levi-Toda solutions identified in 
\eqref{eq:active-profile}. 

\item
Let
\beq\lb{whole mass}
M_i^{(a)}
=
\begin{cases}
\displaystyle
\frac1\pi\int_{\R^2}e^{v_i^{(a)}},
&i\in J_a,\\[6pt]
0,&i\notin J_a.
\end{cases}
\eeq
Then
\begin{equation*}
\begin{aligned}
M_i^{(a)}
&=
\sigma_i(w_{a-1})-\sigma_i(w_a).
\end{aligned}
\end{equation*}
Therefore, the successive bubbles account for the entire
limiting local mass:
\begin{equation*}
\sum_{a=1}^N M_i^{(a)}
=
\langle\omega_i-\tau\omega_i,q_\mu\rangle.
\end{equation*}
At a generic crossing $J_a=\{i_a\}$, this becomes
\[
M_i^{(a)}
=
\delta_{i,i_a}
\langle w_a\alpha_{i_a},q_\mu\rangle.
\]
\end{enumerate}
\end{theorem}

The scales are ordered from larger to smaller radii:
$r_\lambda^{t_{a+1}}/r_\lambda^{t_a}\to0$.
Thus the bubbles indexed by $b>a$ lie inside the $a$th scale.
In actual computations, we start with $w_N=\Id$ and work backward,
determining successively $J_N,w_{N-1},\ldots,J_1,w_0$.
The ray $H+tq_\mu$ records the complete bubbling process.
Its chamber interiors correspond to fast boundaries, its face
crossings determine the bubbling scales and the surviving Levi-Toda
subsystems, and its successive chamber labels record the masses
already enclosed.  The factorization
\(
\tau=\kappa_{J_N}\cdots\kappa_{J_1}
\)
is therefore realized analytically by the successive Toda bubbles,
while the Weyl-group mass formula is the resulting telescoping sum.

The paper is organized as follows. 
Section~\ref{sec:crossings} establishes the chamber-crossing description.
Section~\ref{sec:weights} derives the weight sums and the unified
asymptotic expansion, including the normalized Levi-Toda identification.
Section~\ref{sec:mass} establishes the mass identities and the
uniform concentration bound. Sections~\ref{sec:fast} and~\ref{sec:slow} prove the decay
statements and account for the masses. Section~\ref{sec:A2} gives explicit
$A_2$ examples.

\medskip
\noindent{\bf Acknowledgment of AI assistance.}
As a Plus user, the author used ChatGPT for mathematical discussion, assistance with
checking computations, and editorial suggestions during the preparation
of this paper. The author is responsible for its mathematical content
and final presentation. 

\section{Weyl-chamber crossings}\label{sec:crossings}

We use the standard root-system and Coxeter-group conventions of
\cite{HumphreysCoxeter}. First recall our setup of $H\in \tau C_0$ in \er{def H} and $q_\mu\in C_0$ defined by \er{def q}.  

\begin{proof}[Proof of Part~\textup{(1)} of Theorem~\ref{thm:main}]
For \(\alpha\in\Phi^+\), its root hyperplane is 
$$
\al^\perp = \{X\in \fh_0:\la \al, X\ra = 0\}.
$$
The ray \(X_H(t)=H+tq_\mu\) meets
\(\alpha^\perp\) at
\beq\lb{def ta}
t_\alpha
=
-\frac{\langle\alpha,H\rangle}
       {\langle\alpha,q_\mu\rangle},
\eeq
noting $\la \alpha, q_\mu\ra >0$ by \er{def q}. 
Since \(H=\tau H_0\) and \(H_0\in C_0\),
\[
t_\alpha>0
\iff
\langle\tau^{-1}\alpha,H_0\rangle<0
\iff
\tau^{-1}\alpha\in\Phi^-.
\]
Thus the hyperplanes crossed by the ray $X_H(t)$ for \(t>0\) are precisely
\beq\lb{cross}
\left\{\alpha^\perp:\alpha\in\Inv(\tau)\right\},
\qquad
\Inv(\tau)
=
\{\alpha\in\Phi^+:\tau^{-1}\alpha\in\Phi^-\}.
\eeq
In particular, there are only finitely many crossing times, and the ray doesn't lie in any root hyperplane. 

Denote their distinct values by
\(
0<t_1<\cdots<t_N.
\)

Between two consecutive crossing times, \(X_H(t)\) lies in a unique
open Weyl chamber by 
\[
\fh_0
\setminus
\bigcup_{\alpha\in\Phi^+}\alpha^\perp
=
\bigsqcup_{w\in W} wC_0;
\]
 this defines \(w_a\) by
\[
X_H(t)\in w_aC_0,
\qquad t_a<t<t_{a+1}.
\]
Since \(X_H(0)=\tau H_0\in\tau C_0\), we have \(w_0=\tau\).
Moreover, \(q_\mu\in C_0\), and hence \(X_H(t)\in C_0\) for all
sufficiently large \(t\). Therefore \(w_N=\Id\).

Fix a crossing time \(t_a\). By \er{def w_a}, $X_H(t)$ lies in $w_{a-1}C_0$ immediately before $t_a$, and $w_a C_0$ immediately after $t_a$. 
By \er{fj}, the boundary of $\ol C_0$ is stratified as 
$$
\p \ol{C_0} = \bigsqcup_{\varnothing\ne J\subset\{1,\ldots,n\}} \cf_J.
$$
There is a unique nonempty
\(J_a\subset\{1,\ldots,n\}\) such that
\[
X_H(t_a)\in w_a\mathcal F_{J_a}.
\]
We claim that
\[
w_{a-1}=w_a\kappa_{J_a}.
\]
Indeed, set
\[
Y(t)=w_a^{-1}X_H(t).
\]
For \(t>t_a\) sufficiently close to \(t_a\), we have \(Y(t)\in C_0\),
while \(Y(t_a)\in\mathcal F_{J_a}\). Hence, for every
\(\beta\in\Phi_{J_a}^+\),
\[
\langle\beta,Y(t_a)\rangle=0,
\qquad
\langle\beta,Y(t)\rangle>0\quad (t>t_a).
\]
Since \(\langle\beta,Y(t)\rangle\) is linear in \(t\), it is negative
for \(t<t_a\) sufficiently close to \(t_a\). 
On the other hand, if
\(\beta\in\Phi^+\setminus\Phi_{J_a}^+\), then
\(
\langle\beta,Y(t_a)\rangle>0,
\)
so it remains positive for sufficiently close $t$ on both sides of \(t_a\). 
Thus the positive
roots whose signs are reversed at the crossing are exactly
\(\Phi_{J_a}^+\).

Since
\[
\Inv(\kappa_{J_a})=\Phi_{J_a}^+,
\]
the chamber occupied by \(Y(t)\) immediately before \(t_a\) is
\(\kappa_{J_a}C_0\). Applying $w_{a}$, we see that 
\[
w_{a-1}=w_a\kappa_{J_a}.
\]
Applying this successively and using \(w_0=\tau\) and \(w_N=\Id\), we
obtain
\[
\tau=\kappa_{J_N}\cdots\kappa_{J_1}.
\]

At the $a$th crossing, the crossed hyperplanes are precisely those of
$w_a\Phi_{J_a}^+$. We claim 
$$
w_a\Phi_{J_a}^+\subset \Phi^+.
$$
Indeed, for $\alpha\in w_a\Phi_{J_a}^+$, we have 
$$
\la \alpha, X_H(t)\ra = \la \alpha, H+tq_\mu\ra
$$
changes from negative to positive as $t$ increases past $t_a$. Therefore, 
$$
\la \alpha, q_\mu\ra >0,
$$
and $\alpha\in \Phi^+$ by \er{def q}. 
By \er{cross}, these sets partition $\Inv(\tau)$. Therefore, by \cite[Section~1.7]{HumphreysCoxeter}, 
\[
 \ell(\tau)=|\Inv(\tau)|=\sum_a|\Phi_{J_a}^+|
 =\sum_a\ell(\kappa_{J_a}),
\]
which proves length-additivity \er{leng+}. 

For distinct positive roots $\alpha,\beta$, equality of their
crossing times \er{def ta} imposes a proper linear equation on $H$, since
$\alpha/\langle\alpha,q_\mu\rangle$ and
$\beta/\langle\beta,q_\mu\rangle$ are distinct.
Thus, if $H$ lies outside a finite union of hyperplanes in $\tau C_0$,
all crossing times are distinct.
Hence only one root
hyperplane is crossed at each \(t_a\), so \(J_a=\{j_a\}\) and
\(\kappa_{J_a}=s_{j_a}\). Moreover,
\[
N=|\Inv(\tau)|=\ell(\tau),
\]
the length of $\tau$. 
Consequently,
\[
\tau=s_{j_N}\cdots s_{j_1}
\]
is a reduced expression.
\end{proof}

\begin{remark}[Variation of the bubbling pattern]
For fixed $q_\mu$ and $\tau$, the ordering of the linear functions
$t_\alpha(H)=-\langle\alpha,H\rangle/\langle\alpha,q_\mu\rangle$,
$\alpha\in\Inv(\tau)$, partitions $\tau C_0$ into regions with
fixed crossing order. Their boundaries are given by equal crossing
times. By Theorem~\ref{thm:main}, these regions describe the
separated bubbling patterns, while simultaneous crossings on their
boundaries yield the corresponding Levi bubbles.
\end{remark}

\section{Explicit Toda families and asymptotic expansions}\label{sec:weights}

This section is the technical core of the paper. We introduce the
family $u^{\lambda,H}(r)$ from \cite{IMRN} and its exact weight-sum formulas.
We then derive its asymptotics as $\lambda\to\infty$ in one
statement for $H\in w\mathcal F_J$, including chamber interiors
when $J=\varnothing$. The surviving weights form representations
of the corresponding Levi subalgebra.
We identify the limiting active profile exactly with an explicit
solution of the Levi--Toda subsystem, including its Cartan parameter
and effective singularity parameter.

\subsection{The exact formulas}
Let $G_\C$ be the simply connected complex group with Lie algebra
$\fg$, and let $V_{\omega_i}$ be its $i$th fundamental representation.
Equip it with a Hermitian form $\{\, ,\,\}$, complex linear in the first
argument and invariant under the compact real form $\ck$, and choose a
unit highest-weight vector $v_{\omega_i}$. 
Fix compatible Chevalley generators with
$[e_{\alpha_i},e_{-\alpha_i}]=h_{\alpha_i}$.
In each fundamental representation, their actions satisfy
$e_{\alpha_i}^*=e_{-\alpha_i}$, where $*$ denotes the adjoint
with respect to the chosen Hermitian form.
Let $N_\pm$ be the positive and negative maximal
unipotent subgroups. The construction in \cites{KLNW,IMRN}, based on
\cite{Kostant}, uses the lower-unipotent curve
\[
 \Phi^{-1}\Phi_z=\sum_{i=1}^n z^{\gamma_i}e_{-\alpha_i},
 \qquad \lim_{z\to0}\Phi(z)=\Id
\]
on the slit plane. Although the curve can be multivalued around zero,
the norm formulas below are radial and single-valued.
For $H\in\fh_0$ and $\lambda>0$, set
\begin{equation*}
 T_i^{\lambda,H}(z)
 =\|e^{\lambda H}\Phi(z)v_{\omega_i}\|^2.
\end{equation*}
Writing $A^{-1}=(a^{ij})$ and $\gamma^i=\sum_j a^{ij}\gamma_j$,
the explicit solutions are
\begin{equation}\label{eq:Ui-ui}
\begin{split}
 U_i^{\lambda,H}&=2\gamma^i\log|z|-\log T_i^{\lambda,H},\\
 u_i^{\lambda,H}=\sum_j a_{ij}U_j^{\lambda,H} &= 2\gamma_i\log |z| - \sum_{j} a_{ij}\log T^{\ld, H}_j.
 \end{split}
\end{equation}
They are defined on $\C^*$ and satisfy \eqref{Toda} on the whole
plane, with finite total masses. We study their concentration in $B_2(0)$.
Since $\Phi(z)\to\Id$ at zero, $T_i^{\lambda,H}(0)
=e^{2\lambda\langle\omega_i,H\rangle}$. Consequently, the regular part \er{reg part} has 
\begin{equation}\lb{reg0}
 \hat u_i^{\lambda,H}(0)=-2\lambda\sum_j a_{ij}\langle\omega_j,H\rangle =-2\lambda\langle\alpha_i,H\rangle,
\end{equation}
by \er{raw}. 

For $\tau\ne\Id$ in the Weyl group $W$, $H\in \tau C_0$ has a negative simple-root evaluation, and
$\hat u_i^{\lambda,H}(0)=-2\lambda\langle\alpha_i,H\rangle\to \infty$
for such $i$ as $\ld\to \infty$.
Thus we have a blowup family as in \eqref{blowup}. The main result of \cite{IMRN} is that for $H\in \tau C_0$, we have \er{eq:endpoint-intro}, that is, the blowup mass of the family $u^{\lambda, H}$ is the same as the Weyl-group mass of $\tau$. 

Now we further study the solution family to analyze its blowup picture in more detail. We refer to \cite{Knapp} for the representation theory used below. 

Let $\Wt(V_{\omega_i})$ denote the weight set of $V_{\omega_i}$, and 
choose an orthonormal weight basis $v_{\eta,k}$ of $V_{\omega_i}$,
where $\eta\in \Wt(V_{\om_i})$ and $k$ accounts for its possible multiplicity. Put
\beq\lb{def c}
 c_{i,\eta}=\sum_k
 |\{\Phi(1)v_{\omega_i},v_{\eta,k}\}|^2\geq0,
 \qquad d_i(\eta)=\langle\omega_i-\eta,q_\mu\rangle.
\eeq
Clearly, $c_{i, \eta}$ is independent of the choice of the orthonormal weight basis. 
Every iterated integral lowering $\omega_i$ to $\eta$ has degree
$d_i(\eta)$. Thus the weight-$\eta$ component of
$\Phi(z)v_{\omega_i}$ is $z^{d_i(\eta)}$ times its value at $1$.
Orthogonality of weight spaces gives, for $r=|z|$,
\begin{equation}\label{eq:T-S}
 T_i^{\lambda,H}(r)
 =\sum_{\eta\in \Wt(V_{\omega_i})} c_{i,\eta}e^{2\lambda\langle\eta,H\rangle}
 r^{2d_i(\eta)}
 =r^{2\langle\omega_i,q_\mu\rangle}S_i^{\lambda,H}(r),
\end{equation}
where
\begin{equation}\label{S sum}
 S_i^{\lambda,H}(r)
 =\sum_{\eta\in \Wt(V_{\omega_i})} c_{i,\eta}e^{2\lambda\langle\eta,H\rangle}
 r^{-2\langle\eta,q_\mu\rangle}.
\end{equation}
For $w\in W$, $w\omega_i\in \Wt(V_{\om_i})$, and is called an extremal weight. 
Every extremal weight has multiplicity one, and the nonvanishing
argument in \cite[eq. (3.5)]{IMRN} gives
\begin{equation}\label{eq:extremal-nonzero}
 c_{i,w\omega_i}>0\qquad(w\in W).
\end{equation}
Since $\sum_j a_{ij}\omega_j=\alpha_i$ as in \er{raw} and
$\langle\alpha_i,q_\mu\rangle=1+\gamma_i$ as in \er{def q}, we obtain the exact identity for \er{def f}
\begin{equation}\label{eq:cylindrical-exact}
\begin{split}
 F_i^{\lambda,H}(r)&=u_i^{\lambda,H}(r)+2\log r\\
&= 2\gamma_i \log r - \sum_j a_{ij} \lp 2\la \om_j, q_\mu\ra\log r + \log S_j^{\lambda, H}(r)\rp + 2\log r\\
&= 2\big(\gamma_i - \la \al_i, q_\mu\ra\big)\log r - \sum_j a_{ij}\log S_j^{\lambda, H}(r) + 2\log r\\
&=- \sum_j a_{ij}\log S_j^{\lambda, H}(r). 
\end{split}
\end{equation}

\subsection{Asymptotics on chambers and faces}
As in \er{fj}, we include chamber interiors by allowing $J=\varnothing$, so that
$\mathcal F_{\varnothing}=C_0$. 

Consider the representation $V_{\om_j}$ as a module over $\fg_J$,
the semisimple part of $\fl_J$ \er{lg}, and 
let $V^J_{\omega_j}$ be the submodule generated by $v_{\omega_j}$. 

The highest weight $\om_j$ restricts to a fundamental weight
of the Levi algebra if $j\in J$, and to zero otherwise; for
$J=\varnothing$ the module is one-dimensional and generated by the highest-weight vector.
The following lemma extends the interior observation in
\cite[(3.2)]{IMRN} to all faces.

\begin{lemma}[The exposed weight face]\label{lem:face-weights}
Fix $w\in W$, $J\subset\{1,\ldots,n\}$, and let
$H\in w\mathcal F_J$ as in \er{fj}. Then
\begin{gather}
 \langle w\omega_j,H\rangle\geq \langle\eta,H\rangle,
 \qquad \forall\eta\in\Wt(V_{\omega_j}),\lb{max}\\
 \mathcal E_j(H)
 :=\{\eta\in\Wt(V_{\omega_j}):
 \langle\eta,H\rangle=\langle w\omega_j,H\rangle\}
 =w\Wt(V^J_{\omega_j}).\lb{id wts}
\end{gather}
More precisely, choose a representative $\bar w$ of $w$ in $\ck$, the compact real form of $G_\bc$, acting on
$V_{\omega_j}$. For every $\nu\in \Wt(V^J_{\omega_j})$, 
\begin{equation}\label{eq:face-weight-spaces}
 V_{\omega_j}[w\nu]
 =\bar w\bigl(V^J_{\omega_j}[\nu]\bigr),
\end{equation}
where brackets denote weight spaces for the ambient Cartan algebra.

In particular,
\begin{equation}\label{eq:inactive-face}
 \mathcal E_j(H)=\{w\omega_j\}\quad\text{if }j\notin J.
\end{equation}
If $J=\{i\}$, then
$\mathcal E_i(H)=\{w\omega_i,w(\omega_i-\alpha_i)\}$.
\end{lemma}

\begin{proof}
Put $Y=w^{-1}H\in\mathcal F_J$. For any
$\eta\in\Wt(V_{\omega_j})$, write
\[
 \omega_j-w^{-1}\eta=\sum_i b_i\alpha_i,
 \qquad b_i\in\mathbb Z_{\ge0}.
\]
By \eqref{fj}, $\langle\alpha_i,Y\rangle$ vanishes for $i\in J$
and is positive for $i\notin J$. Hence
\begin{gather}
 \langle\omega_j-w^{-1}\eta,Y\rangle\ge0,\notag\\
 \langle\omega_j-w^{-1}\eta,Y\rangle=0
 \quad\Longleftrightarrow\quad
 \omega_j-w^{-1}\eta
 \in\sum_{i\in J}\mathbb Z_{\ge0}\alpha_i.\lb{equal}
\end{gather}
The first inequality and Weyl invariance of the pairing gives \eqref{max}.

We now identify the corresponding weight spaces. By the
Poincar\'e--Birkhoff--Witt theorem and the highest-weight property,
\beq\lb{pull down}
 V_{\omega_j}=U(\mathfrak n_-)v_{\omega_j},
 \qquad
 V^J_{\omega_j}=U(\mathfrak n_-^J)v_{\omega_j},
\eeq
where $\fn_-$ and $\fn_-^J$ are the negative nilpotent Lie
subalgebras of $\fg$ and $\fg_J$, respectively, and $U$ denotes
the universal enveloping algebra.
Since $\mathfrak n_-$ is generated by the negative simple-root
vectors $f_i=e_{-\alpha_i}$, the weight space $V_{\omega_j}[\nu]$
is spanned by the nonzero words
\beq\lb{act}
 f_{i_1}\cdots f_{i_m}v_{\omega_j},
 \qquad
 \alpha_{i_1}+\cdots+\alpha_{i_m}=\omega_j-\nu.
\eeq

We now show that
\beq\lb{spell}
\Wt(V_{\om_j}^J)
=\left\{\nu\in\Wt(V_{\om_j}):
\omega_j-\nu\in\sum_{i\in J}\mathbb Z_{\ge0}\alpha_i\right\}.
\eeq
The $\subset$ direction is clear.
For the other direction, \er{act} expresses $\om_j-\nu$ as
a sum of simple roots $\al_{i_r}$, while the condition on the
right-hand side of \er{spell} expresses it using only simple
roots indexed by $J$. Linear independence of the simple roots
forces every $i_r$ to belong to $J$. By \er{pull down},
such a $\nu$ belongs to $\Wt(V_{\om_j}^J)$.
\er{equal} and \er{spell} then prove \er{id wts}.

Furthermore, every spanning word in \er{act} for such a $\nu\in \Wt(V_{\om_j}^J)$
uses only indices in $J$, so
\beq\lb{id wt sp}
 V_{\omega_j}[\nu]=V^J_{\omega_j}[\nu].
\eeq
Then
\[
V_{\om_j}[w\nu]
=\bar w(V_{\om_j}[\nu])
=\bar w(V_{\om_j}^J[\nu]),
\]
proving \er{eq:face-weight-spaces}.

If $j\notin J$, each $f_i$ with $i\in J$ annihilates
$v_{\omega_j}$, since
$\langle\omega_j,\alpha_i^\vee\rangle=0$.
Hence $V^J_{\omega_j}$ is the highest-weight line, proving
\eqref{eq:inactive-face}. If $J=\{i\}$ and $j=i$, then
$\langle\omega_i,\alpha_i^\vee\rangle=1$, so the module
$V^{\{i\}}_{\omega_i}$ has precisely the two weights
$\omega_i$ and $\omega_i-\alpha_i$.
\end{proof}

Set
\begin{equation}\label{def P}
 P_j^H(r)=\sum_{\eta\in\mathcal E_j(H)}
 c_{j,\eta}r^{-2\langle\eta,q_\mu\rangle}.
\end{equation}
These sums never vanish: by \eqref{eq:extremal-nonzero},
\begin{equation}\label{def P-positive}
 P_j^H(r)\geq c_{j,w\omega_j}
 r^{-2\langle w\omega_j,q_\mu\rangle}>0\qquad(r>0).
\end{equation}

The next proposition 
is our main asymptotic result. 

\begin{proposition}[Asymptotics on chambers and faces]\label{bound asym}
For $H\in w\mathcal F_J$, with $J$ possibly empty, the exact formula is
\begin{equation}\label{eq:boundary-factor}
 S_j^{\lambda,H}(r)=e^{2\lambda \la w\omega_j, H\ra }P_j^H(r)
 (1+E_j^{\lambda,H}(r)),
\end{equation}
where
\[
 E_j^{\lambda,H}(r)=\frac{1}{P_j^H(r)}
 \sum_{\eta\notin\mathcal E_j(H)}c_{j,\eta}
 e^{-2\lambda( \la w\omega_j, H\ra-\langle\eta,H\rangle)}
 r^{-2\langle\eta,q_\mu\rangle}.
\]
Consequently,
\begin{equation}\label{eq:boundary-F}
 F_i^{\lambda,H}(r)
 =-2\lambda\langle\alpha_i,w^{-1}H\rangle
 -\sum_j a_{ij}\log P_j^H(r)+O_{\cc}(e^{-\delta\lambda}),
\end{equation}
where $\delta>0$ is independent of $r$. 
The estimate holds with all $D=r\partial_r$-derivatives on any compact
$\cc\Subset(0,\infty)$. 
\end{proposition}

\begin{proof}
In \er{S sum}, by \er{max} and \er{id wts}, we see that the terms with the maximal exponential factor sum to 
$$
e^{2\lambda \la w\omega_j, H\ra }P_j^H(r).
$$
Factoring this out gives \er{eq:boundary-factor}.

If $J\ne\{1,\ldots,n\}$, put
$\delta=\min_{i\notin J}\langle\alpha_i,w^{-1}H\rangle>0$.
For $\eta\notin\mathcal E_j(H)$, the proof of
Lemma~\ref{lem:face-weights} gives
\[
\langle w\omega_j,H\rangle-\langle\eta,H\rangle\geq\delta.
\]
If $J=\{1,\ldots,n\}$, no weights are omitted and the remainder
vanishes identically; any $\delta>0$ may be used.

\eqref{def P-positive} bounds the denominator away from zero
on compact annuli. This proves exponential decay of $E_j$ and all
its derivatives there. Substitution in
\eqref{eq:cylindrical-exact} gives \eqref{eq:boundary-F}; its exact
error is $-\sum_j a_{ij}\log(1+E_j^{\lambda,H})$.
\end{proof}

For $J=\varnothing$, the face sums reduce to
$P_j^H(r)=c_{j,w\omega_j}r^{-2\langle w\omega_j,q_\mu\rangle}$.
Thus Proposition~\ref{bound asym} also gives the interior expansion
\begin{equation*}
 F_i^{\lambda,H}(r)
 =-2\lambda\langle\alpha_i,w^{-1}H\rangle
 +2\langle w\alpha_i,q_\mu\rangle\log r
 -\sum_j a_{ij}\log c_{j,w\omega_j}
 +O_{\cc}(e^{-\delta\lambda}).
\end{equation*}

When $\lambda\to \infty$, we look at the asymptotics of $F_i^{\lambda, H}(r)$. 
If $H\in w\cf_J$, then by \er{fj} and in \er{eq:boundary-F}, we have
\begin{equation}\lb{good lim}
F_i^{\lambda, H}(r)\longrightarrow
\begin{cases}
-\infty, & i\notin J,\\
-\sum_j a_{ij} \log P_j^H(r), & i\in J.
\end{cases}
\end{equation}

\subsection{Identification of the face sums with Levi sums}

Next, we identify the finite limits in \er{good lim} with the Levi version of \er{eq:cylindrical-exact}, including their coefficients. 

We use the decomposition
\begin{equation}\label{eq:levi-cartan-split}
 \fh_0=\fh_{J,0}\oplus\fz_{J,0},
\end{equation}
where 
\beq\lb{hz}
 \fh_{J,0}=\operatorname{span}_{\R}\{h_{\alpha_i}:i\in J\},\quad
 \fz_{J,0}=\{Z\in\fh_0:\langle\alpha_i,Z\rangle=0,\ i\in J\}.
\eeq
(The second summand is the real form of the center \er{cx ct}.)

All weights in $\wtjj$ differ by integral combinations of the $\al_i$ for $i\in J$, so 
 have the same restriction
$\omega_j|_{\fz_{J,0}}$ to the center, by \er{hz}. 

Put $f=\sum_{i=1}^n e_{-\alpha_i}$. For regular $q\in\fh_0$,
let $\psi_q\in N_-$ be the unique element satisfying
$\psi_q^{-1}q\psi_q=q-f$; see \cite[Lemma~3.5.1]{Kostant}.
For the Levi algebra $\fg_J$, use the same construction with
$f_J=\sum_{i\in J}e_{-\alpha_i}$. For regular $q_J\in\fh_{J,0}$
and $K^J\in\fh_{J,0}$, the Levi sums
$S_{j,J}^{1,K^J}(r;q_J)$ are defined by the weight-sum formula
\er{S sum} for $V^J_{\omega_j}$, using the canonical element
for $q_J$ in place of $\Phi(1)$. For $j\notin J$, set
$S_{j,J}^{1,K^J}(r;q_J)=1$.

\begin{proposition}\label{prop:P-S}
Let $H\in w\mathcal F_J$, put 
\begin{equation}\lb{had it}
q^{\mathrm{eff}}=w^{-1}q_\mu=q_J+q_Z
\end{equation}
according to
\eqref{eq:levi-cartan-split}. There exist $K^J\in\fh_{J,0}$ and
$K_Z\in\fz_{J,0}$ such that, for $r>0$ and $1\le j\le n$,
\[
P_j^H(r)=e^{2\langle\omega_j,K_Z\rangle}
 r^{-2\langle\omega_j,q_Z\rangle}S_{j,J}^{1,K^J}(r;q_J).
\]
Consequently, for $i\in J$,
\[
\sum_{j=1}^n a_{ij}\log P_j^H(r)
=\sum_{j\in J}a_{ij}\log S_{j,J}^{1,K^J}(r;q_J).
\]
The parameter $q_J$ is regular; its dominance at a crossing will be
established in \eqref{eq:active-positive}. 
\end{proposition}

\begin{proof}
Put $\psi=\Phi(1)$. 
By \eqref{def P} and Lemma~\ref{lem:face-weights},
\beq\lb{direct P}
 P_j^H(r)
 =\sum_{\nu\in\Wt(V^J_{\omega_j})}
   c_{j,w\nu}\,r^{-2\langle w\nu,q_\mu\rangle}
   =\sum_{\nu\in\Wt(V^J_{\omega_j})}
   c_{j,w\nu}\,r^{-2\langle\nu,q^{\mathrm{eff}}\rangle}.
\eeq
Choose a compact representative $\bar w$ of $w$ and, for each
$\nu\in\Wt(V^J_{\omega_j})$, an orthonormal basis
$\{v_{\nu,k}\}_k$ of $V^J_{\omega_j}[\nu]$.
By \eqref{eq:face-weight-spaces}, the vectors
$\{\bar wv_{\nu,k}\}_k$ form an orthonormal basis of the entire
ambient weight space $V_{\omega_j}[w\nu]$. Hence
$$
 c_{j,w\nu}
 =\sum_k
   \bigl|\{\psi v_{\omega_j},\bar wv_{\nu,k}\}\bigr|^2
 =\sum_k
   \bigl|\{\bar w^{-1}\psi v_{\omega_j},v_{\nu,k}\}\bigr|^2.
   $$

The algebraic characterization in \cite[Proposition~2.13]{IMRN} gives
$\psi=\Phi(1)=\psi_{q_\mu}$. Since $\bar w$ is compact,
$\bar w^*=\bar w^{-1}$. The proof of \cite[Proposition~2.19]{IMRN},
following \cite[Proposition~8.1]{BKK}, gives
\begin{equation}\label{eq:weyl-gauss}
 \bar w^{-1}\psi=h\psi_1\widetilde y,
 \qquad
 \psi_1=\psi_{q^{\mathrm{eff}}},\quad
 \widetilde y\in N_+,
\end{equation}
where $h$ belongs to the complex Cartan subgroup.
Since $\widetilde yv_{\omega_j}=v_{\omega_j}$, and $h$ acts
on the weight space of weight $\nu$ by the character
$\chi_\nu(h)$, we obtain
\begin{equation}\label{eq:face-coefficients-before-projection}
 c_{j,w\nu}
= |\chi_\nu(h)|^2
   \sum_k\bigl|\{\psi_1v_{\omega_j},v_{\nu,k}\}\bigr|^2
 =e^{2\langle\nu,K\rangle}
   \sum_k\bigl|\{\psi_1v_{\omega_j},v_{\nu,k}\}\bigr|^2,
\end{equation}
where $K\in\mathfrak h_0$ is determined by $h^*h=e^{2K}$ such that $|\chi_\nu(h)|^2=e^{2\langle\nu,K\rangle}$.

We next replace $\psi_1$ by its Levi projection.
The negative root spaces outside $\Phi_J$ form an ideal in
$\mathfrak n_-$. The corresponding quotient homomorphism
\[
 \pi_J:N_-\longrightarrow N_{J,-}
\]
allows us to write
\[
 \psi_J=\pi_J(\psi_1),
 \qquad
 \psi_1=n_{\mathrm{out}}\psi_J,
 \qquad n_{\mathrm{out}}\in\ker\pi_J.
\]
Every nonidentity term in the action of $n_{\mathrm{out}}$
lowers a weight by a sum of positive roots having a nonzero
simple-root coefficient outside $J$. Since every weight of
$\psi_Jv_{\omega_j}$ differs from $\omega_j$ by a combination
of simple roots in $J$, these terms have weights outside
$\Wt(V^J_{\omega_j})$; compare the proof of
Lemma~\ref{lem:face-weights}. Orthogonality of distinct weight
spaces therefore gives
\beq\lb{1j}
 \{\psi_1v_{\omega_j},v_{\nu,k}\}
 =\{\psi_Jv_{\omega_j},v_{\nu,k}\},
 \qquad \nu\in\Wt(V^J_{\omega_j}).
\eeq

Define
\beq\lb{def d}
 d_{j,\nu}
 :=\sum_k\bigl|\{\psi_Jv_{\omega_j},v_{\nu,k}\}\bigr|^2.
\eeq
By \er{1j}, equation \eqref{eq:face-coefficients-before-projection} now becomes
\[
 c_{j,w\nu}=e^{2\langle\nu,K\rangle}d_{j,\nu}.
\]

By \er{eq:weyl-gauss}, 
\[
 \psi_1^{-1}q^{\mathrm{eff}}\psi_1
 =q^{\mathrm{eff}}-f.
\]
Its projection onto the negative Borel algebra of the Levi algebra yields
\begin{equation*}
 \psi_J^{-1}q_J\psi_J=q_J-f_J,
 \qquad f_J=\sum_{i\in J}e_{-\alpha_i}.
\end{equation*}
Here the projection sends $q^{\mathrm{eff}}$ to $q_J$ and
kills the negative root spaces outside $\Phi_J$.
For $\beta\in\Phi_J$, we have
$\langle\beta,q_J\rangle=\la \beta, q^{\mathrm{eff}}\ra = \langle w\beta,q_\mu\rangle\ne0$ since $q_\mu\in C_0$. 
Thus $q_J$ is regular for the Levi root system, and uniqueness
identifies $\psi_J$ as its canonical element.

Consequently, \er{direct P} becomes 
\begin{equation}\label{eq:face-coefficient-sum}
 P_j^H(r)
 =\sum_{\nu\in\Wt(V^J_{\omega_j})}
  d_{j,\nu}\, e^{2\langle\nu,K\rangle} 
   r^{-2\langle\nu,q^{\mathrm{eff}}\rangle}.
\end{equation}

Using \eqref{eq:levi-cartan-split}, write
\begin{equation*}
K=K^J+K_Z,\qquad K^J\in\fh_{J,0},\quad K_Z\in\fz_{J,0}.
\end{equation*}
For every $\nu\in\Wt(V^J_{\omega_j})$ and
$Z\in\mathfrak z_{J,0}$, we have
\[
 \langle\nu,Z\rangle=\langle\omega_j,Z\rangle.
\]
Separating these common central factors in
\eqref{eq:face-coefficient-sum} gives
\beq\lb{exp P}
 P_j^H(r)
 =
 e^{2\langle\omega_j,K_Z\rangle}
 r^{-2\langle\omega_j,q_Z\rangle}
 \sum_{\nu\in\Wt(V^J_{\omega_j})}
 d_{j,\nu}\,e^{2\langle\nu,K^J\rangle}
 r^{-2\langle\nu,q_J\rangle}.
\eeq
For $j\in J$, define the Levi sum by
\[
S_{j,J}^{1,K^J}(r;q_J)
=\sum_{\nu\in\Wt(V^J_{\omega_j})}
 d_{j,\nu}e^{2\langle\nu,K^J\rangle}r^{-2\langle\nu,q_J\rangle}.
\]
This is the sum analogous to \er{S sum} for the Levi subalgebra $\fg_J$ with 
its coefficient \er{def d} the counterpart of \er{def c}. 

For $j\notin J$, Lemma~\ref{lem:face-weights} gives
$V^J_{\omega_j}=\mathbb C v_{\omega_j}$, so the sum in \er{exp P} has only one term for $\nu=\omega_j$. 
Then the coefficient $d_{j, \omega_j} = 1$ by \er{def d}, and 
from 
$\omega_j|_{\mathfrak h_{J,0}}=0$, we see this term is $1$.
Set $S_{j,J}^{1,K^J}=1$ for $j\notin J$. 

We obtain for all $j$ between 1 and $n$, 
\begin{equation*}
P_j^H(r)=e^{2\langle\omega_j,K_Z\rangle}
 r^{-2\langle\omega_j,q_Z\rangle}S_{j,J}^{1,K^J}(r;q_J).
\end{equation*}

Finally, for $i\in J$ and $Z\in\mathfrak z_{J,0}$,
\[
 \sum_{j=1}^n a_{ij}\langle\omega_j,Z\rangle
 =\langle\alpha_i,Z\rangle=0.
\]
Therefore,
\begin{align*}
\sum_{j=1}^n a_{ij}\log P_j^H(r)
&=2\sum_j a_{ij}\la\omega_j,K_Z\ra
  -2\log r\sum_j a_{ij}\la\omega_j,q_Z\ra\\
&\qquad+\sum_j a_{ij}\log S_{j,J}^{1,K^J}(r;q_J)\\
&=\sum_{j\in J}a_{ij}\log S_{j,J}^{1,K^J}(r;q_J).
\end{align*}
\end{proof}

\section{Logarithmic differentiation, probabilities, and mass}\label{sec:mass}
In this section, we will prove Theorem \ref{just 0} for the uniform selection bound. 
Fix $\lambda>0$ and $H\in\fh_0$. We omit them from the notation and
write, for example, $u_i(r)$ for $u_i^{\lambda,H}(r)$ and $F_i(r)$ for
$F_i^{\lambda,H}(r)$. Along the way, we interpret the enclosed mass
$\sigma_i(u;r)$ as a weighted average of the quantities $d_i(\eta)$
defined in \eqref{def c}.

The quantity $e^{F_i(r)}$ is the mass density in logarithmic radius:
\beq\lb{rate}
 e^{u_i(r)}r\,dr=e^{F_i(r)}\,d\log r.
\eeq
With $s=\log r$, the operator $D=r\partial_r$ introduced in Proposition \ref{bound asym} 
is $\partial_s$.
For a radial function $f$, direct differentiation gives
\[
 D^2f=r^2f_{rr}+rf_r=r^2\Delta f.
\]
On $\C^*$, the equation for $U_i$ in \eqref{eq:Ui-ui} 
$$
\Delta U_i + 4e^{u_i} = 0
$$
therefore gives
\begin{equation}\label{eq:radial-density}
 \Delta\log T_i=4e^{u_i},
 \qquad D^2\log S_i=4r^2e^{u_i}
 =4e^{F_i}.
\end{equation}
Here by \er{eq:T-S}, $\log T_i-\log S_i=2\langle\omega_i,q_\mu\rangle\log r$
has zero second logarithmic derivative.

We now differentiate the finite sum \eqref{S sum}
explicitly. 
For $\eta\in \Wt(V_{\om_i})$, abbreviate
\[
 b_\eta=\langle\eta,q_\mu\rangle,\qquad
 A_\eta(r)=c_{i,\eta}e^{2\lambda\langle\eta,H\rangle}r^{-2b_\eta},
 \qquad S_i(r)=\sum_\eta A_\eta(r).
\]
Since $DA_\eta=-2b_\eta A_\eta$, we have
\[
 DS_i=-2\sum_\eta b_\eta A_\eta,
 \qquad D^2S_i=4\sum_\eta b_\eta^2A_\eta.
\]
Normalize the nonnegative summands by setting
\begin{equation*}
 p_{i,\eta}(r)=\frac{A_\eta(r)}{S_i(r)},
 \qquad \sum_\eta p_{i,\eta}(r)=1.
\end{equation*}
They define a probability distribution on the finite weight set $\Wt(V_{\om_i})$.
With $p_\eta=p_{i,\eta}(r)$ and
$\bar b=\sum_\eta p_\eta b_\eta$, the preceding derivatives yield
\begin{align}
 D\log S_i&=\frac{DS_i}{S_i}=-2\bar b,
 \label{eq:mean-slope}\\
 D^2\log S_i
 &=\frac{D^2S_i}{S_i}-\left(\frac{DS_i}{S_i}\right)^2
 =4\left(\sum_\eta p_\eta b_\eta^2-\bar b^2\right).
 \label{eq:second-moment}
\end{align}
Combining \eqref{eq:radial-density} and \eqref{eq:second-moment},
we obtain
\begin{equation}\label{eq:variance}
 e^{F_i(r)}
 =\sum_\eta p_{i,\eta}(r)b_\eta^2
 -\left(\sum_\eta p_{i,\eta}(r)b_\eta\right)^2,
\end{equation}
Thus the cylindrical density is the variance of the weight evaluations $b_\eta$.
The probability language here is a clean way to describe these
normalized positive sums.

Relatedly, the enclosed mass has an equally direct interpretation. Integrating the equation 
$$\Delta U_i+4e^{u_i}=4\pi\gamma^i\delta_0$$ over $B_r$ gives
\[
 2\pi DU_i(r)+4\int_{B_r}e^{u_i}=4\pi\gamma^i,
\]
where we have used
$$
\int_{B_r}\Delta U_i = \int_{S_r}\frac{\p U_i}{\p r} = \frac{1}{r} DU_i \cdot 2\pi r = 2\pi DU_i(r).
$$
Using \eqref{eq:Ui-ui}, \eqref{eq:T-S}, and \eqref{eq:mean-slope},
we find
\begin{equation}\label{eq:radial-mass}
 \begin{aligned}
 \sigma_i(u;r)
 &=\gamma^i-\tfrac12DU_i(r)
 =\tfrac12D\log T_i(r)\\
 &=\langle\omega_i,q_\mu\rangle+\tfrac12D\log S_i(r)\\
 &=\langle\omega_i,q_\mu\rangle-\bar b
 =\sum_\eta p_{i,\eta}(r)d_i(\eta).
 \end{aligned}
\end{equation}
This identity holds at every radius, before any limit is taken:
the actual enclosed mass is the weighted average of the fixed
values $d_i(\eta)=\langle\omega_i-\eta,q_\mu\rangle$ \er{def c} attached to
the weights.

\begin{remark} Differentiating \er{eq:radial-mass} with respect to $s$ and using the easily verified
$$
\frac{dp_{i,\eta}}{ds} = 2(\bar b - b_\eta) p_{i,\eta},
$$
we obtain $d\sigma_i/ds=2\operatorname{Var}(b_\eta)=2e^{F_i}$, recovering \er{eq:variance} by \er{rate} and the normalization of $\sigma_i$. 
\end{remark}

Let $\kappa$ be the longest element of the full Weyl group. The
largest and smallest $b_\eta$ in $V_{\omega_i}$ are
\begin{equation*}
 B_i=\langle\omega_i,q_\mu\rangle,
 \qquad L_i=\langle\kappa\omega_i,q_\mu\rangle.
\end{equation*}
Indeed, $\omega_i-\eta$ is a nonnegative integral combination of
simple roots, so $B_i\geq \langle\eta,q_\mu\rangle$.
Applying this fact to the weight $\kappa\eta$ and then applying
$\kappa$ shows that $\eta-\kappa\omega_i$ is a nonnegative
combination of simple roots; hence $\langle\eta,q_\mu\rangle\geq L_i$.
Strict positivity of the simple-root evaluations of $q_\mu$ makes
both extremizers unique.

\begin{proof}[Proof of Theorem~\ref{just 0}]
We use the standard variance bound for a bounded random variable. 

Since $b_\eta$ lies in $[L_i,B_i]$,
\[
 (b_\eta-L_i)(B_i-b_\eta)\geq0,
 \qquad b_\eta^2\leq(L_i+B_i)b_\eta-L_iB_i.
\]
Multiplying by $p_\eta$ and summing, then using
\eqref{eq:variance}, yields
\begin{align*}
 e^{F_i}
 &=\sum_\eta p_\eta b_\eta^2-\bar b^2\\
 &\leq(L_i+B_i)\bar b-L_iB_i-\bar b^2\\
 &=(B_i-\bar b)(\bar b-L_i)\\
 &\leq \frac{(B_i-L_i)^2}{4} = \frac{\mathcal M_i^2}{4}.
\end{align*}
The endpoints $L_i,B_i$ depend only on $q_\mu$ and the representation;
they do not depend on $H$, $\lambda$, or $r$. This proves all the
asserted uniformity.
\end{proof}

\section{Fast boundaries and enclosed masses}\label{sec:fast}

The following exact rescaling identity converts the exponential
scales into translation of the Cartan parameter. Together with
Sections~\ref{sec:crossings} and~\ref{sec:weights}, it determines
both the fast boundaries and the slow bubbling scales. 

\begin{proposition}\lb{key id}
For the scaling 
\begin{equation}\label{ansatz}
 r_\lambda^t=e^{-\lambda t},\qquad t\geq0,
\end{equation}
we have the exact rescaling identities
\begin{align}
 S_j^{\lambda,H}(e^{-\lambda t}\rho)
 &=S_j^{\lambda,X_H(t)}(\rho),\qquad X_H(t)=H+tq_\mu,
 \label{eq:S-log-time}\\
 u_i^{\lambda,H}(e^{-\lambda t}\rho)-2\lambda t
 &=u_i^{\lambda,X_H(t)}(\rho).
 \label{eq:scaled-field-exact}
\end{align}
\end{proposition}

\begin{proof}
In each term of \eqref{S sum}, we see that 
$$
c_{i,\eta}e^{2\lambda\langle\eta,H\rangle}
 (e^{-\lambda t}\rho)^{-2\langle\eta,q_\mu\rangle} = 
 c_{i,\eta}e^{2\lambda\langle\eta,H+tq_\mu\rangle}
 \rho^{-2\langle\eta,q_\mu\rangle},
 $$
and this proves \er{eq:S-log-time}. 

Now, this and \eqref{eq:cylindrical-exact} give
$$ u_i^{\lambda, H}(e^{-\lambda t}\rho) + 2\log (e^{-\lambda t}\rho) 
= u_i^{\lambda, X_H(t)}(\rho) + 2\log(\rho),
$$
hence proving \er{eq:scaled-field-exact}. 
\end{proof}

\begin{proof}[Proof of Part~\textup{(2)} of Theorem~\ref{thm:main}]
Fix $t\in(t_a,t_{a+1})$ and put $\rho=|y|$. Since
$X_H(t)\in w_aC_0$, \eqref{eq:scaled-field-exact} and the
$J=\varnothing$ case of Proposition~\ref{bound asym} give
\begin{multline*}
 u_i^{\lambda,H}(e^{-\lambda t}y)-2\lambda t = u_i^{\lambda, X_H(t)}(\rho)\\
 =F_i^{\lambda, X_H(t)}(\rho) - 2\log\rho = -2\lambda\langle\alpha_i,w_a^{-1}X_H(t)\rangle+O_{\cc}(1)
 \longrightarrow-\infty
\end{multline*}
on each compact $\cc\Subset\C^*$ for all $i$, since $w_a^{-1}X_H(t)\in C_0$. 

The case of $y=1$ shows that $u_i^{\lambda, H}$ has fast decay at \er{ansatz} for $t\in (t_a, t_{a+1})$. 

Now we study the mass using the general formula \er{eq:radial-mass} at $r_\lambda^t = e^{-\lambda t}$. 
Then 
$$
\sigma_i(u^{\lambda, H}; e^{-\lambda t}) = \sum_{\eta\in \Wt(V_{\om_i})} p_{i,\eta}^{\lambda, H}(e^{-\lambda t})d_i(\eta),
$$
and 
$$
p_{i,\eta}^{\lambda, H}(e^{-\lambda t}) = \frac{c_{i,\eta}e^{2\lambda\la\eta, X_H(t)\ra}}{\sum_\eta c_{i,\eta}e^{2\lambda\la\eta, X_H(t)\ra}}.
$$
Since $X_H(t)\in w_aC_0$, by Lemma \ref{lem:face-weights} for $J=\emptyset$, only $\eta=w_a\om_i$ maximizes the exponent $\la \eta, X_H(t)\ra$. Therefore, as $\lambda\to \infty$, we have
$$
\lim_{\lambda\to\infty}\sigma_i(u^{\lambda,H};e^{-\lambda t})
=d_i(w_a\omega_i)=\sigma_i(w_a),
$$
by \er{def c} and \er{mass for w}. 

These estimates are uniform for $t$ in compact subintervals of
$(t_a,t_{a+1})$. In particular, the mass between two such fixed
logarithmic times tends to zero. 
\end{proof}

\section{Slow decay and Levi--Toda bubbles}\label{sec:slow}

\begin{proof}[Proof of Parts~\textup{(3)} and \textup{(4)} of
Theorem~\ref{thm:main}] 
In this part, we fix a crossing time $t_a$, put $J=J_a$, $w=w_a$, and write 
$$
X_*=H+t_aq_\mu\in w\cf_J,\quad Y=w^{-1}X_*\in \cf_J.
$$


We write \er{had it} for $w=w_a$ as 
$$
q^{\mathrm{eff}}=w_a^{-1}q_\mu=q^{(a)}+q_Z.
$$
For $t>t_a$ sufficiently close to $t_a$,
\[
 w^{-1}X_H(t)=w^{-1}(X_*+(t-t_a)q_\mu) = Y+(t-t_a)q^{\mathrm{eff}}\in C_0.
\]
Since $\langle\alpha_i,Y\rangle=0$ for $i\in J$, this gives
\begin{equation}\label{eq:active-positive}
 \langle\alpha_i,q^{\mathrm{eff}}\rangle=\la w\al_i, q_\mu\ra>0\qquad(i\in J).
\end{equation}

By 
\eqref{eq:scaled-field-exact} and \eqref{eq:boundary-F}, the rescaled solution $\widetilde u_i^{\lambda,a}(y)$ 
in \er{a-scale} becomes
\begin{equation*}
\begin{split} 
 \widetilde u_i^{\lambda,a}(y) &= u_i^{\lambda, X_*}(\rho) = F_i^{\lambda, X_*}(\rho) - 2\log\rho\\
 &=-2\lambda\langle\alpha_i,Y\rangle-2\log\rho
 -\sum_ja_{ij}\log P_j^{X_*}(\rho)+O_{\cc}(e^{-\delta\lambda}),
 \end{split}
\end{equation*}
with all derivatives on compact subsets $\cc\subset\C^*$.
The fields are radial, so the logarithmic derivative estimates on
compact radial intervals imply these Cartesian derivative estimates.
For $i\notin J$, $\langle\alpha_i,Y\rangle>0$ and 
\beq\lb{not j}
\widetilde u_i^{\lambda, a}(y) \to -\infty,\quad i\notin J.
\eeq
For $i\in J$, $\langle\alpha_i,Y\rangle=0$ and, by Proposition \ref{prop:P-S}, the limit is
\begin{equation}\label{eq:active-profile}
 \begin{aligned}
 v_i^{(a)}(\rho)
 &=-2\log\rho-\sum_{j=1}^na_{ij}\log P_j^{X_*}(\rho)\\
 &=-2\log\rho-\sum_{j\in J}a_{ij}
 \log S_{j,J}^{1,K^{(a)}}(\rho;q^{(a)}),
 \end{aligned}
\end{equation}
for some $K^{(a)}\in \fh_{J_a,0}$. 
Thus $v^{(a)}(\rho)$, including its
coefficients, is an explicit Levi--Toda solution
for $\fg_{J_a}$ as in \er{eq:cylindrical-exact}. By \eqref{eq:active-positive}, $q^{(a)}$ is
dominant for the Levi root system. Thus the canonical construction
gives a finite-energy solution with singularity strengths $\gamma_i^{(a)}$ such that 
\[
\gamma_i^{(a)} + 1=\langle\alpha_i,q^{(a)}\rangle
=\la \al_i, w_a^{-1}q_\mu\ra = \langle w_a\alpha_i,q_\mu\rangle>0,
\qquad i\in J,
\]
This verifies the Levi-Toda system 
\eqref{eq:levi-main} with its strength of singularity \er{strength}. 
If $J$ is disconnected, the construction applies
separately to each simple factor. This proves Part~\textup{(3)}.

In particular, for $y=1$, $F_i^{\lambda,H}(e^{-\lambda t_a})$ converges to
$v_i^{(a)}(1)$ for $i\in J$, proving the stated slow decay.

Now apply \cite[Theorem~1.1]{KLNW} to the explicit Levi solution in
\eqref{eq:active-profile}, which continues to hold even if $J$ is disconnected. 
The result says that the mass on the whole plane is the Weyl-group mass \er{mass for w} with respect to the current singularity strength for the longest Weyl group element. Therefore the $M_i^{(a)}$ in \er{whole mass} is 
\begin{equation*}
 \begin{aligned}
 M_i^{(a)}
 &=\langle\omega_i-\kappa_J\omega_i,q^{\mathrm{(a)}}\rangle\overset{\dagger}=\langle\omega_i-\kappa_J\omega_i, w_a^{-1}q_\mu\rangle\\
& =\langle w_a(\omega_i-\kappa_J\omega_i),q_\mu\rangle\\
 &=\la \om_i-w_a\kappa_J\om_i, q_\mu\ra - \la \om_i-w_a\om_i,q_\mu\ra \\
 &= \sigma_i(w_{a-1})-\sigma_i(w_a).
 \end{aligned}
\end{equation*}
The equality $\dagger$ holds since the difference $w_a^{-1}q_\mu - q^{(a)}\in \fz_{J, 0}$ and 
the difference $\omega_i-\kappa_J \om_i$ belongs to the root lattice of $\fg_J$. 
The last equality uses $w_{a-1}=w_a\kappa_J$ in \er{wwk}. 


For $i\notin J$, the Levi-Weyl group $W_J$ fixes $\omega_i$ and both
sides are zero. Since $\sigma_i(w_N)=\sigma_i(\Id)=0$, telescoping gives
\[
 \sum_{b=a}^NM_i^{(b)}=\sigma_i(w_{a-1}),\qquad
 \sum_{a=1}^NM_i^{(a)}=\sigma_i(\tau).
\]
This proves Part~\textup{(4)}. For $J=\{i_a\}$, the identity
$\omega_i-s_{i_a}\omega_i=\delta_{i,i_a}\alpha_{i_a}$
gives the stated Liouville mass.
\end{proof}

\begin{remark}
The Levi equations themselves already follow from
the much easier \er{not j}: the inactive fields tend to
$-\infty$, and passing to the limit in \eqref{Toda} gives
\[
 \Delta v_i^{(a)}+4\sum_{j\in J_a}a_{ij}e^{v_j^{(a)}}=0
 \quad\text{on }\mathbb C^*,\qquad i\in J_a.
\]
Proposition~\ref{prop:P-S} and \er{eq:active-profile} identify the particular explicit
solution, including its Cartan parameter and all coefficients.
\end{remark}

\begin{remark}
The singularity parameters \er{strength} satisfy
\begin{equation*}
\begin{split}
 \gamma_i^{(a)}
& =\langle\alpha_i,q^{\mathrm{(a)}}\rangle-1
=\la w_a \al_i, q_\mu\ra - 1 = \la \al_i, q_\mu\ra - \la \al_i-w_a\al_i,q_\mu\ra-1\\
&= \mu_i - \sum_{j} a_{ij}\la \om_j -w_a\om_j, q_\mu\ra - 1 = \gamma_i - \sum_j a_{ij} \sigma_j(w_a)>-1,
 \qquad i\in J.
 \end{split}
\end{equation*}
Thus the effective source subtracts the masses of the smaller bubbles
already enclosed. At the innermost scale $w_N=\Id$, it equals the
original source on the active components.

The term $\sum_j a_{ij}\sigma_j(w_a)$ records the contribution of
the smaller-scale bubbles, whose masses concentrate at the origin
under the present rescaling.

Indeed, their contribution to $4\sum_j a_{ij}e^{u_j}$ in
\eqref{Toda} converges to
$4\pi\sum_j a_{ij}\sigma_j(w_a)\delta_0$, which moves to the
right-hand side in the limiting equation. 
\end{remark}

\begin{remark}[The first bubbling scale]
For the nonsingular Toda system, the usual choice
\[
\varepsilon^\lambda=e^{-M^\lambda/2},
\qquad M^\lambda=\max_i\sup_{B_2(0)}u_i^\lambda,
\]
makes the rescaled fields
$\widetilde u_i^\lambda(y)
=u_i^\lambda(\varepsilon^\lambda y)-M^\lambda$ in \er{scaling} 
nonpositive on the rescaled domain.

In our singular setting, the rescaling \eqref{scaling}
transforms the regular parts \eqref{reg part} by
\[
\widehat{\widetilde u}_i^\lambda(y)
=\hat u_i^{\lambda,H}(\varepsilon^\lambda y)
 +2\mu_i\log\varepsilon^\lambda.
\]
Thus, to make their maximum over the components at the origin
equal to zero, we choose
\[
\varepsilon^\lambda
=\exp\left(-\frac12\max_i
\frac{\hat u_i^{\lambda,H}(0)}{\mu_i}\right)
=\exp\left(-\lambda\max_i
\frac{-\la \al_i, H\ra}{\mu_i}\right),
\]
by \er{reg0}. 
The division by $\mu_i$ accounts for the different scaling rates
of the regular parts.

For $\tau\ne\Id$, this is exactly the innermost scale
$e^{-\lambda t_N}$ in \er{aftall}, the first bubbling scale when proceeding
outward from the origin. Indeed, $H+tq_\mu\in C_0$ precisely
when $\langle\alpha_i,H\rangle+t\mu_i>0$ for every $i$, so
\[
t_N=\max_i\frac{-\langle\alpha_i,H\rangle}{\mu_i}.
\]

Starting from this normalized innermost scale, the subsequent
bubbling scales are determined by the other intersections of
the ray $H+tq_\mu$ with root hyperplanes, taken in decreasing
order of $t$ as the radius increases.
\end{remark}

\section{Explicit \texorpdfstring{$A_2$}{A2} examples}\label{sec:A2}

Let
\[
 A=\begin{pmatrix}2&-1\\-1&2\end{pmatrix},\qquad
 \mu_1,\mu_2>0,\qquad m=\mu_1+\mu_2.
\]
We first examine one bubble and its central factors explicitly, then
compare separated crossings with a simultaneous crossing.

\subsection{The weight sums}
In the standard representation of $\SL_3$, take
$e_{-\alpha_1}=E_{21}$ and $e_{-\alpha_2}=E_{32}$. The defining
initial-value problem gives
\[
 \Phi(z)=\begin{pmatrix}
 1&0&0\\
 z^{\mu_1}/\mu_1&1&0\\
 z^m/(\mu_2m)&z^{\mu_2}/\mu_2&1
 \end{pmatrix}.
\]
Indeed, differentiating this matrix gives
$\Phi_z=\Phi(z^{\gamma_1}E_{21}+z^{\gamma_2}E_{32})$.
Write $H=\diag(h_1,h_2,h_3)$ and
$q_\mu=\diag(q_1,q_2,q_3)$, with trace zero, so
$q_1-q_2=\mu_1$ and $q_2-q_3=\mu_2$.
Using $\C^3$ and $\bigwedge^2\C^3$ for the fundamental
representations, we obtain
\begin{align}
 S_1^{\lambda,H}(r)
 &=e^{2\lambda h_1}r^{-2q_1}
 +\frac{e^{2\lambda h_2}}{\mu_1^2}r^{-2q_2}
 +\frac{e^{2\lambda h_3}}{\mu_2^2m^2}r^{-2q_3},
 \label{eq:A2-S1}\\
 S_2^{\lambda,H}(r)
 &=e^{2\lambda(h_1+h_2)}r^{-2(q_1+q_2)}
 +\frac{e^{2\lambda(h_1+h_3)}}{\mu_2^2}r^{-2(q_1+q_3)}
  \label{eq:A2-S2}
 \\
 &\hspace{14pt}
 +\frac{e^{2\lambda(h_2+h_3)}}{\mu_1^2m^2}r^{-2(q_2+q_3)}.
\notag
\end{align}
The two fields are
$u_1=-2\log r-2\log S_1+\log S_2$ and
$u_2=-2\log r+\log S_1-2\log S_2$.

\subsection{One bubble and the role of the center}\label{sec:A2-center-check}
Before listing the general crossing patterns, we work out one bubble.
Take
\[
 \mu_1=2,\quad\mu_2=3,\quad
 q_\mu=\operatorname{diag}(7/3,1/3,-8/3),\quad
 H=\operatorname{diag}(-3,1,2).
\]
At $t=1$ we have $X_H(1)= H + 1\cdot q_\mu=\operatorname{diag}(-2/3,4/3,-2/3)$.
Thus $w=s_1$ and $J=\{2\}$: indeed,
$w^{-1}X_H(1)=\operatorname{diag}(4/3,-2/3,-2/3)$, so $\alpha_2$ vanishes on it. 

At the scale $r=e^{-\lambda}\rho$, the exposed terms in
\eqref{eq:A2-S1}--\eqref{eq:A2-S2} are
\[
 P_1(\rho)=\frac14\rho^{-2/3},\qquad
 P_2(\rho)=\rho^{-16/3}+\frac1{100}\rho^{14/3}.
\]
This follows since the three powers of $e^{2\lambda}$ in $S_1$ of \er{eq:A2-S1} are 
$$
h_1+ q_1 = -\frac{2}{3},\quad h_2+q_2=\frac{4}{3},\quad h_3+q_3=-\frac{2}{3},
$$
so $P_1$ records the coefficient of the second term with highest power as 
$$
\frac{1}{\mu_1^2}\rho^{-2q_2} = \frac{1}{4} \rho^{-2/3}.
$$
Similarly, we get $P_2$ from $S_2$ as the coefficients of the first and the last term in \er{eq:A2-S2}.

The rescaled fields therefore satisfy
\[
 \begin{aligned}
 \widetilde u_1^\lambda
 &=-4\lambda-2\log\rho-2\log P_1+\log P_2+o(1)
 \longrightarrow-\infty,\\
 \widetilde u_2^\lambda
 &=-2\log\rho+\log P_1-2\log P_2+o(1)
 \longrightarrow v_2.
 \end{aligned}
\]
In fact,
\[
 \frac{P_2}{\sqrt{P_1}}=2\rho^{-5}+\frac1{50}\rho^5,
 \qquad
 v_2=8\log\rho-2\log\left(2+\frac{\rho^{10}}{50}\right).
\]
We now recover these coefficients from the Gauss decomposition, that is, we identify the $P_i$ with suitable $S_i$.

Write $\psi=\Phi(1)$ and choose the compact representative
\[
 \psi=\begin{pmatrix}1&0&0\\1/2&1&0\\1/15&1/3&1\end{pmatrix},
 \qquad
 \bar w=\begin{pmatrix}0&1&0\\-1&0&0\\0&0&1\end{pmatrix}.
\]
Direct multiplication verifies
\[
 \bar w^{-1}\psi=h\psi_1n_+,\qquad
 h=\operatorname{diag}(-1/2,-2,1),
\]
\[
 \psi_1=\begin{pmatrix}1&0&0\\-1/2&1&0\\1/15&1/5&1\end{pmatrix},\quad
 n_+=\begin{pmatrix}1&2&0\\0&1&0\\0&0&1\end{pmatrix}.
\]
Here $q'=w^{-1}q_\mu=\operatorname{diag}(1/3,7/3,-8/3)$ and
$\psi_1^{-1}q'\psi_1=q'-E_{21}-E_{32}$.
For $J=\{2\}$ the usual Levi projection on $N_-$ is
\[
 \pi_J\begin{pmatrix}1&0&0\\a&1&0\\b&c&1\end{pmatrix}
 =\begin{pmatrix}1&0&0\\0&1&0\\0&c&1\end{pmatrix}.
\]
Consequently
\[
 \psi_J=\pi_J(\psi_1)
 =\begin{pmatrix}1&0&0\\0&1&0\\0&1/5&1\end{pmatrix}.
\]
The entry is $1/5$, not the $1/3$ in the original $\psi$:
we project the transformed unipotent factor. The Cartan decomposition is
\[
 q'=q_J+q_Z,
 \qquad q_J=\operatorname{diag}(0,5/2,-5/2),\quad
 q_Z=\operatorname{diag}(1/3,-1/6,-1/6).
\]
Thus $\psi_J^{-1}q_J\psi_J=q_J-E_{32}$, as required.

The Cartan factor requires a separate decomposition. From
$h^*h=e^{2K}$ we obtain
\[
 \begin{aligned}
 K&=\operatorname{diag}(-\log2,\log2,0)=K^J+K_Z,\\
 K^J&=\operatorname{diag}(0,\tfrac12\log2,-\tfrac12\log2),\\
 K_Z&=\operatorname{diag}(-\log2,\tfrac12\log2,\tfrac12\log2).
 \end{aligned}
\]
Both $q_Z$ and $K_Z$ have equal second and third entries, so
$\alpha_2$ vanishes on them. Hence
they lie in the center of this Levi algebra.

We verify the profile in both the $T$- and $S$-formulations.
The effective singularity strength is $\gamma_2^{(a)}=5-1=4$.
Restricting $K^J$ to its active two-dimensional block gives
$\diag(\frac12\log2,-\frac12\log2)$. With $\rho=|z|$, we obtain
\[
\begin{aligned}
v_2(\rho)
&=8\log\rho-2\log\left\|
\begin{pmatrix}\sqrt{2}&0\\0&1/\sqrt{2}\end{pmatrix}
\begin{pmatrix}1&0\\z^5/5&1\end{pmatrix}
\begin{pmatrix}1\\0\end{pmatrix}\right\|^2\\
&=8\log\rho-2\log\left(2+\frac{\rho^{10}}{50}\right),
\end{aligned}
\]
in agreement with the profile above.

The corresponding Levi sum is
\[
 S^{1,K^J}(\rho;q_J)
 =2\rho^{-5}+\frac{1}{2\cdot25}\rho^5,
\]
so that 
$$
2\log P_2 - \log P_1 = 2\log S.
$$

The scalar equation and its mass are
\[
 \Delta v_2+8e^{v_2}=16\pi\delta_0,\qquad
 \frac1\pi\int_{\mathbb R^2}e^{v_2}=5,
\]
which agree with $\langle w\alpha_2,q_\mu\rangle=\mu_1+\mu_2=5$.

\subsection{Three separated Liouville bubbles}
We now demonstrate the bubble decomposition in Theorem~\ref{thm:main}. 

Choose $p,q>0$ with
\[
 \langle\alpha_1,H\rangle=-p,\qquad
 \langle\alpha_2,H\rangle=-q,\qquad
 \frac{p}{\mu_1}<\frac{q}{\mu_2}.
\]
Then $H\in\kappa C_0$, where
$\kappa=s_2s_1s_2=s_1s_2s_1$ is the longest element of $W(A_2)$.
The roots crossed, in increasing logarithmic time, are
$\alpha_1,\alpha_1+\alpha_2,\alpha_2$, at
\[
 t_1=\frac{p}{\mu_1},\qquad
 t_2=\frac{p+q}{m},\qquad
 t_3=\frac{q}{\mu_2}.
\]
Their order follows because $t_2$ is the strictly intermediate
weighted average of $t_1$ and $t_3$. Starting with $w_0=\kappa$, each
crossed root $\beta_a$ determines $j_a$ by
$\alpha_{j_a}=-w_{a-1}^{-1}\beta_a$, and then
$w_a=w_{a-1}s_{j_a}$. The chamber and bubble data are
\[
\begin{array}{c|c|c|c|c}
 a&\beta_a=w_a\alpha_{j_a}&J_a&w_a&(M_1^{(a)},M_2^{(a)})\\ \hline
 1&\alpha_1&\{2\}&s_2s_1&(0,\mu_1)\\
 2&\alpha_1+\alpha_2&\{1\}&s_2&(m,0)\\
 3&\alpha_2&\{2\}&\Id&(0,\mu_2)
\end{array}
\]
Since $r_a=e^{-\lambda t_a}$ decreases as $a$ increases, we compute
the masses from outside in. The successive enclosed mass vectors are
\[
 \sigma(w_0)=(m,m),\quad \sigma(w_1)=(m,\mu_2),\quad
 \sigma(w_2)=(0,\mu_2),\quad \sigma(w_3)=(0,0).
\]
The differences $M_i^{(a)}=\sigma_i(w_{a-1})-\sigma_i(w_a)$ give
the bubble masses in the table. Notice that the index of the surviving component
is the simple-root index in the chamber after the crossing, rather
than the index of the root hyperplane in the original coordinates.
For example, crossing $\alpha_1^\perp$ first produces a bubble in
component $2$.

\subsection{A simultaneous crossing}
In the previous example, if $p/\mu_1=q/\mu_2=c>0$, then $H=-cq_\mu$ and
\[
 X_H(t)=(t-c)q_\mu.
\]
All three positive-root hyperplanes are crossed at $t=c$.
There is one block, $J=\{1,2\}$, with $w_0=\kappa$ and
$w_1=\Id$. Every weight is retained, and the effective parameter
is $q_\mu$ itself.

In this case the rescaling is exact for every $\lambda$:
\[
 S_j^{\lambda,-cq_\mu}(e^{-\lambda c}r)
 =S_j^{1,0}(r).
\]
The choice $\lambda=1$ on the right is immaterial because its
Cartan parameter is zero. Here $X_H(c)=0$, so both leading terms
$-2\lambda\langle\alpha_i,X_H(c)\rangle$ vanish.
Consequently the two rescaled fields are independent of $\lambda$.
Writing
\[
 \begin{aligned}
 T_1^{1,0}(r)&=1+\frac{r^{2\mu_1}}{\mu_1^2}
 +\frac{r^{2m}}{\mu_2^2m^2},\\
 T_2^{1,0}(r)&=1+\frac{r^{2\mu_2}}{\mu_2^2}
 +\frac{r^{2m}}{\mu_1^2m^2}.
 \end{aligned}
\]
the full $A_2$ bubble is
\[
 \begin{aligned}
 v_1&=2\gamma_1\log r-2\log T_1^{1,0}+\log T_2^{1,0},\\
 v_2&=2\gamma_2\log r+\log T_1^{1,0}-2\log T_2^{1,0}.
 \end{aligned}
\]
Its mass vector is again $(m,m)$. Thus the same endpoint mass is
realized either by three scalar bubbles at separated scales or by
one full $A_2$ bubble at a simultaneous crossing.

\end{document}